\documentclass[12pt]{amsart}
\usepackage{amsfonts}
\usepackage[letterpaper, left=2.5cm, right=2.5cm, top=2.5cm,
bottom=2.5cm,dvips]{geometry}
\newtheorem{theorem}{Theorem}
\theoremstyle{plain}

\newtheorem{conjecture}{Conjecture}
\newtheorem{corollary}{Corollary}

\numberwithin{equation}{section}

\begin{document}
\title[Majority choosability of digraphs]{Majority choosability of digraphs}
\author{Marcin Anholcer}
\address{Faculty of Informatics and Electronic Economy, Pozna\'{n}
	University of Economics and Business,  61-875 Pozna\'{n}, Poland}
\email{m.anholcer@ue.poznan.pl}
\author{Bart\l omiej Bosek}
\address{Theoretical Computer Science Department, Faculty of Mathematics and Computer Science, Jagiellonian
University, 30-348 Krak\'{o}w, Poland}
\email{bosek@tcs.uj.edu.pl}
\author{Jaros\l aw Grytczuk}
\address{Faculty of Mathematics and Information Science, Warsaw University
	of Technology, 00-662 Warsaw, Poland}
\email{j.grytczuk@mini.pw.edu.pl}
\dedicatory{Dedicated to Micha{\l} Karo\'{n}ski on the ocassion of his 70th birthday.}
\thanks{Research partially supported by the National Science Center of Poland, grant 2011/03/B/ST6/01367.}

\begin{abstract}
A \emph{majority coloring} of a digraph is a coloring of its vertices such that for each vertex $v$, at most half of the out-neighbors of $v$ has the same color as $v$. A digraph $D$ is \emph{majority $k$-choosable} if for any assignment of lists of colors of size $k$ to the vertices there is a majority coloring of $D$ from these lists. We prove that every digraph is majority $4$-choosable.  This gives a positive answer to a question posed recently by Kreutzer, Oum, Seymour, van der Zypen, and Wood in \cite{Kreutzer}. We obtain this result as a consequence of a more general theorem, in which majority condition is profitably extended. For instance, the theorem implies also that every digraph has a coloring from arbitrary lists of size three, in which at most $2/3$ of the out-neighbors of any vertex share its color. This solves another problem posed in \cite{Kreutzer}, and supports an intriguing conjecture stating that every digraph is majority $3$-colorable.
\end{abstract}

\maketitle

\section{Introduction}

Let $D$ be a directed graph. Let $d^{+}(v)$ denote the number of out-neighbors of vertex $v$. A coloring $c$ of the vertices of $D$ is called \emph{majority coloring} if for every vertex $v$ the number of its out-neighbors in color $c(v)$ is at most $\frac{1}{2}d^+(v)$. This concept was introduced recently by van der Zypen \cite{Zypen}, in connection to neural networks, and studied by Kreutzer, Oum, Seymour, van der Zypen, and Wood in \cite{Kreutzer}. It is proved there, among other results, that every digraph is majority $4$-colorable. The proof is very simple: first, notice that every digraph with no directed cycles is majority $2$-colorable (just apply greedy coloring), next, split the edges of a given digraph into two acyclic digraphs, and take the product coloring. It is conjectured in \cite{Kreutzer} that actually three colors are sufficient for majority coloring of any digraph. This would be best possible since a majority coloring of an odd directed cycle must be a proper coloring of the underlying undirected graph.

Another interesting problem posed in \cite{Kreutzer} concerns \emph{list} version of the majority coloring. Suppose that each vertex $v$ of a digraph $D$ is assigned with a list of colors $L(v)$. Then $D$ is majority colorable \emph{from} these lists if there is a majority coloring $c$ of $D$ with $c(v)\in L(v)$. If $D$ is majority colorable from any lists of size $k$, then we say that $D$ is \emph{majority $k$-choosable}. The authors of \cite{Kreutzer} asked if there is a finite number $k$ such that every digraph is majority $k$-choosable. We answer this question in the affirmative by proving a more general theorem which implies that actually every digraph is majority $4$-choosable. As another consequence we infer that every digraph is $3$-choosable so that at most $\frac{2}{3}d^+(v)$ of the out-neighbors of any vertex $v$ have the same color as $v$. This solves another problem posed in \cite{Kreutzer}, and extends a result of Seymour from \cite{Seymour}, asserting that every digraph has $3$-coloring in which at least one out-neighbor of each vertex (of positive out-degree) is colored differently.

There are many variants of majority coloring that may be studied in a variety of contexts (see \cite{Kreutzer}). Perhaps our approach might be useful in some of these situations. We shall discuss briefly these issues in the final section.

\section{The results}
Our main result reads as follows.

\begin{theorem}
Let $D$ be a directed graph. Suppose that each vertex $v$ is assigned with a list $L(v)$ of four colors. Suppose further that each color $x$ in $L(v)$ is assigned with a real number $r_{v}(x)$, the \emph{rank} of color $x$ in $L(v)$. Assume that for every vertex $v$, the color ranks $r_{v}(x)$ satisfy the following condition: 
\begin{equation} \label{dupa} \tag{$\ast$} \sum_{x\in L(v)}r_{v}(x) \geq 2d^{+}(v).
\end{equation}
Then there is a vertex coloring of $D$ from lists $L(v)$ satisfying the following constraint: If $x$ is a color assigned to $v$, then the number of out-neighbors of $v$ in color $x$ is at most $r_{v}(x)$.

\end{theorem}

\begin{proof}
Let us remark first that we do not impose any restrictions on color ranks, except condition (\ref{dupa}). These ranks may be positive, negative, or zero. If $r_v(x)=0$ and $v$ is colored with $x$, then to satisfy the assertion of the theorem, none of the  out-neighbors of $v$ may be colored with $x$. If $r_v(x)$ is strictly negative, then actually $v$ cannot be colored with $x$ at all (no set may have negative cardinality).

The proof goes by induction on the number of vertices in $D$. It is not hard to check that the theorem is true for one-vertex digraph. Indeed, by condition (\ref{dupa}), at least one color rank in the list must be non-negative, and we may use it to color the only vertex in the digraph. So, let $n\geq 2$, and assume that the assertion of the theorem is true for all digraphs with at most $n-1$ vertices. Let $D$ be a digraph on $n$ vertices satisfying the assumptions of the theorem, and let $v$ be any vertex of $D$. Consider a new digraph $D'$ obtained by deleting vertex $v$ with color ranks modified as follows. Let $a$ and $b$ be the two colors with highest ranks, $r_{v}(a)$ and $r_{v}(b)$, in the list $L(v)$. For each in-coming neighbor $u$ of vertex $v$, decrease the ranks $r_{u}(a)$ and $r_{u}(b)$ by one, provided these colors are contained in the list $L(u)$. All the remaining color ranks in these or other lists are left unchanged.

We claim that digraph $D'$ with modified color ranks still satisfies condition (\ref{dupa}). Indeed, for each in-coming neighbor $u$ of $v$, the left hand side of (\ref{dupa}) decreased by at most two, while the right-hand side of (\ref{dupa}) decreased by exactly two (since the out-degree $d^{+}(u)$ decreased by exactly one). So, by the inductive assumption there is a coloring of $D'$ satisfying the assertion of the theorem.

We now extend this coloring to the deleted vertex $v$ in the following way. First notice that
\begin{equation} \label{pipa} \tag{$1$} r_{v}(a)+r_{v}(b) \geq d^{+}(v).
\end{equation}
Indeed, by the maximality of ranks of colors $a$ and $b$ in the list $L(v)$, the inequality $r_{v}(a)+r_{v}(b) < d^{+}(v)$ would imply $\sum_{x\in L(v)}r_{v}(x) < 2d^{+}(v)$, contrary to the assumption. Let $n_{a}$ and $n_b$ denote the number of out-neighbors of $v$ colored with colors $a$ and $b$, respectively. Obviously, $n_{a}+n_{b}\leq d^{+}(v)$. Hence, by (\ref{pipa}), at least one of the following inequalities must be satisfied:
\begin{equation} \label{chuj} \tag{$2$} r_{v}(a)\geq n_{a} \quad \text{or} \quad r_{v}(b) \geq n_{b}.
\end{equation}
We chose a color whose rank satisfies one of these inequalities, and assign that color to $v$.

We claim that the extended coloring satisfies the assertion of the theorem. First, let $u$ be arbitrary in-coming neighbor of $v$. Let $x$ denote the color assigned to $u$ in coloring of $D'$. If $x$ is one of the colors $a$ or $b$, then the number of out-neighbors of $u$ in $D'$ colored with $x$ is at most $r_{u}(x)-1$, by inductive assumption. Thus, their number in $D$ after coloring the vertex $v$ is still bounded by $r_{u}(x)$. If $x$ is neither equal to $a$ nor to $b$, then the constraint is fulfilled even more. If $u$ is an arbitrary out-neighbor of $v$, or any other vertex of $D'$, then the corresponding constraint holds by induction, since out-neighborhoods and color ranks for such vertices remained unchanged in $D'$. Finally, for the vertex $v$ we have chosen color $a$ or $b$ so that the corresponding inequality of (\ref{chuj}) is satisfied. This completes the proof.

\end{proof}

We obtain now easily the aforementioned consequences for majority choosability of digraphs.

\begin{corollary}
Every digraph is majority $4$-choosable.
\end{corollary}
\begin{proof}
Put $r_{v}(x)=\frac{1}{2}d^{+}(v)$ for each vertex $v$ and for every color $x$ from its list $L(v)$, and apply the theorem.
\end{proof}

\begin{corollary}
Let $D$ be a digraph with color lists of size three assigned to the vertices. Then there is a coloring from these lists such that for each vertex $v$, at most $\frac{2}{3}$ of its out-neighbors have the color of $v$.
\end{corollary}

\begin{proof}
Let $0<\varepsilon<\frac{1}{3}$ be a real number. Let $v$ be a vertex in $D$, and let $L(v)$ denote its list with three colors. For each color $x$ in $L(v)$ assign the rank $r_{v}(x)=\frac{2}{3}d^{+}(v)+\varepsilon$. Now, add a new fictitious color $f$ with the rank $r_{v}(f)=-3\varepsilon$ to each list $L(v)$. The assertion of the corollary follows now directly from Theorem 1.
\end{proof}

\section{Discussion}
There are many variants of majority coloring that may be studied for various combinatorial structures (see \cite{Kreutzer}). For instance, in a multi-color version considered in \cite{Kreutzer}, the majority constraint is strengthened to $\frac{1}{k}d^+(v)$, where $k\geq2$ is a fixed integer. It is easy to see that $k$ colors are sufficient for acyclic digraphs, and thus $k^2$ colors suffice for arbitrary digraph (by product coloring). It is conjectured in \cite{Kreutzer} that $k+1$ colors are actually enough. As noted by David Wood (personal communication), the proof of Theorem 1 can be easily extended to the multi-color setting, however, it only gives the same quadratic bound in the list version of the problem.

The situation looks much simpler for undirected graphs. An old result of Lov\'{a}sz \cite{Lovasz} asserts that every graph is majority $2$-colorable, and more generally, it is $k$-colorable so that at most $1/k$ neighbors of each vertex share its color, for every $k \geq2$. The proof is very simple: just take a coloring that minimizes the total number of monochromatic edges. The same argument works in the list version, and after slight modification it gives a result similar to Theorem 1 (with color ranks in each list summing up to at least the degree of the corresponding vertex).
 
Majority coloring may be studied for infinite graphs as well (see \cite{Aharoni}). For undirected graphs it is known as the problem of \emph{unfriendly partitions} (see \cite{Garden}). As proved by Shelah and Milner \cite{Shelah}, every infinite graph is majority $3$-colorable, but there are graphs on uncountably many vertices that are not majority $2$-colorable. Whether every countably infinite graph has a majority $2$-coloring remains a mystery. Perhaps it would be interesting to consider similar questions for infinite directed graphs.

We conclude the paper with the following strengthening of the majority coloring conjecture from \cite{Kreutzer}.
\begin{conjecture}
Every digraph is majority $3$-choosable.
\end{conjecture}

\end{document}